\documentclass[acmsmall, review=false,screen=true,authorversion=true, usenames, svgnames]{acmart}
\pdfoutput=1

\setcopyright{none}
\acmDOI{}

\usepackage{multirow}
\usepackage[utf8]{inputenc}
\usepackage[T1]{fontenc}
\usepackage[australian,american]{babel}
\usepackage{graphicx}
\usepackage{thmtools, mathdots}
\usepackage[framemethod=tikz]{mdframed}
\usepackage{pgfplots}
\pgfplotsset{compat=1.12}
\usepackage{array}
\usepackage[binary-units]{siunitx}
\usepackage{booktabs}
\usepackage{multirow}
\usepackage{pifont}
\usepackage{tikz}
\usepackage{csquotes}
\usepackage{natbib}
\setcitestyle{square,aysep={},yysep={;}}
\usepackage{color}
\usepackage{todonotes}
\usepackage{lipsum}
\usepackage{subfigure}
\usepackage{pgfplotstable}
\usepackage{verbatim}

\newcommand{\N}{\mathbb{N}}
\newcommand{\Z}{\mathbb{Z}} 
\newcommand{\Q}{\mathbb{Q}} 

\newcommand{\NN}{\mathbb{N}} 
\newcommand{\ZZ}{\mathbb{Z}}

\renewcommand{\epsilon}{\varepsilon} 
\renewcommand{\phi}{\varphi}
\renewcommand{\emph}{\textbf}

\newcommand{\BS}{BS(1,m)}
\newcommand{\BStwo}{BS(1,m^2)}
\DeclareMathOperator{\diam}{diam}
\DeclareMathOperator{\cay}{Cay}
\DeclareMathOperator{\ord}{ord}
\DeclareMathOperator{\lcm}{lcm}

\newboolean{showcomments}
\setboolean{showcomments}{true}
\ifthenelse{\boolean{showcomments}}
{ \newcommand{\mynote}[3]{
    \fbox{\bfseries\sffamily\scriptsize#1}
    {\small$\blacktriangleright$\textsf{\textit{\color{#3}{#2}}}$\blacktriangleleft$}}}
{ \newcommand{\mynote}[3]{}}

\begin{abstract}
  For $0<\alpha\le 1$, we say that a sequence $(X_k)_{k>0}$ of $d$-regular graphs has property $D_\alpha$ if
  there exists a constant $C>0$ such that $\diam(X_k)\ge C\cdot|X_k|^\alpha$. We investigate property
  $D_\alpha$ for arithmetic box spaces of the solvable Baumslag-Solitar groups $BS(1,m)$ (with $m\geq 2$):
  those are box spaces obtained by embedding $BS(1,m)$ into the upper triangular matrices in $GL_2(\Z[1/m])$
  and intersecting with a family $M_{N_k}$ of congruence subgroups of $GL_2(\Z[1/m])$, where the levels $N_k$
  are coprime with $m$ and $N_k|N_{k+1}$. We prove:
  \begin{itemize}
  \item if an arithmetic box space has $D_\alpha$, then $\alpha\le\frac{1}{2}$~;
  \item if the family $(N_k)_k$ of levels is supported on finitely many primes, the corresponding arithmetic
    box space has $D_{1/2}$~;
  \item if the family $(N_k)_k$ of levels is supported on a family of primes with positive analytic primitive
    density, then the corresponding arithmetic box space does not have $D_\alpha$, for every $\alpha>0$.
\end{itemize}
Moreover, we prove that if we embed $BS(1,m)$ in the group of invertible upper-triangular matrices
$T_n(\Z[1/m])$, then every finite index subgroup of the embedding contains a congruence subgroup. This is a
version of the \textit{congruence subgroup property} (CSP).
\end{abstract}

\begin{document}

\title{On arithmetic properties of solvable Baumslag-Solitar groups}

\author{Laurent Hayez}
\affiliation{%
  \institution{Université de Neuchâtel}
  \city{Neuchâtel}
  \country{Switzerland}
}
\email{laurent.hayez@unine.ch}
\authornote{Supported by grant 200021\_188578 of the Swiss National Fund for Scientific Research}

\author{Tom Kaiser}
\affiliation{%
  \institution{Université de Neuchâtel}
  \city{Neuchâtel}
  \country{Switzerland}
}
\email{tomkaiser456@gmail.com}

\author{Alain Valette}
\affiliation{%
  \institution{Université de Neuchâtel}
  \city{Neuchâtel}
  \country{Switzerland}
}
\email{alain.valette@unine.ch}

\maketitle

\section{Introduction}
\label{sec:introduction}

Let $G$ be a finitely generated, residually finite group. If $(H_k)_{k>0}$ is a decreasing sequence of finite
index normal subgroups of $G$, with trivial intersection, and $S$ is a finite generating set of $G$, then the {\it
  box space} $\square_{(H_k)}G$ is the disjoint union of finite Cayley graphs
$$\square_{(H_k)}G=\coprod_{k>0} \cay (G/H_k,S);$$
here by abuse of notation we identify $S$ with its image in $G/H_k$. Changing the generating set $S$ does not
change the coarse geometry of the box space\footnote{In the sense that the two families of graphs are
  quasi-isometrically equivalent, by a family of quasi-isometries with uniform constants.} so we omit $S$ from the notation.

In the dictionary between group-theoretical properties of $G$ and metric properties of $\square_{(H_k)}G$ (see
e.g. \cite{BourAna}), it is natural to look at the behaviour of the diameter of the Cayley graphs $\cay(G/H_k,S)$. Let $0<\alpha\leq 1$. The box space $\square_{(H_k)}G$ satisfies {\it property $D_\alpha$} if there is some constant $C>0$ such that for every $k>0$:
\begin{equation}
  \label{eq:property_D_alpha}
  \diam(\cay(G/H_k,S))\geq C|G/H_k|^\alpha.
\end{equation}

Note that property $D_\alpha$ is a coarse geometry invariant of the box space. The following is known. 

\begin{theorem} 
  \label{thm:property_D_alpha_and_cyclic_groups}
  Let $G$ be a finitely generated, residually finite group.
\begin{enumerate}
\item (see Cor. 1.7 and Lemma 5.1 in \cite{BreuillardTointon}) If some box space of $G$ has property $D_\alpha$, for some $\alpha>0$, then $G$ virtually maps onto $\Z$.
\item (see Theorem 3 in \cite{khukhro2017expanders}) If $G$ maps onto $\Z$, then for every $0<\alpha<1$, there exists a box space of $G$ with property $D_\alpha$.
\item (see Proposition 5 in \cite{khukhro2017expanders}) The group $G$ is virtually cyclic if and only if some (hence any) box space of $G$ has property $D_1$.
\end{enumerate}
\hfill$\square$
\end{theorem}

This paper considers the \textit{Baumslag-Solitar groups} $BS(n,m)\;(m,n>0)$ with presentation
$$BS(n,m)=\langle a,t|ta^nt^{-1}=t^m\rangle.$$
They all map onto $\Z$, by $a\mapsto 0,t\mapsto 1$. On the other hand they are known to be residually finite
if and only if $n=1$ or $n=m$: see Theorem C in \cite{Meskin}. It turns out that the solvable Baumslag-Solitar groups $BS(1,m)$, with $m\geq 2$, have
interesting box spaces. Indeed it is well-known that $BS(1,m)$ may be viewed as a semi-direct product
$$BS(1,m)=\Z[1/m]\rtimes\Z$$
where the factor $\Z$ corresponds to the subgroup $\langle t \rangle$ acting by powers of $m$. We may identify this semi-direct product with the following subgroup $G_m$ of
upper triangular matrices in $GL_2(\Z[1/m])$:
$$G_m=\left\{\left(\begin{array}{cc}m^k & r \\0 & 1\end{array}\right):k\in\Z,\;r\in \Z[1/m]\right\}.$$
The isomorphism is obtained by mapping
$a$ to $A = \left(\begin{array}{cc}1 & 1 \\0 & 1\end{array}\right)$ and $t$ to $ T = \left(\begin{array}{cc}m &
    0 \\0 & 1\end{array}\right)$. The associated embedding of $BS(1,m)$ into $GL_2(\Z[ 1/m])$ is called the {\it
  standard embedding}.

\medskip
In $GL_n(\Z[1/m])$ we may define congruence subgroups. Let $N>0$ be coprime with $m$. The {\it
  principal congruence subgroup of level $N$} is the kernel $M_N$ of the reduction modulo $N$:
$$M_N=\ker\left[GL_n(\Z[1/m])\rightarrow GL_n(\Z/N\Z)\right].$$

\begin{definition} If $G$ is any subgroup of $GL_n(\Z[1/m])$, and $N> 0$ is coprime to $m$, then the \textit{congruence subgroup} $G(N)$ in $G$ is
$$G(N):=G\cap M_N.$$
\end{definition}

For a sequence of integers such
that each one divides the next one, one obtains a sequence of nested congruence subgroups, and thus a box space of $BS(1,m)$. Such box spaces deserve to be called {\it
  arithmetic box spaces}. We will study property $D_\alpha$ for the arithmetic box spaces of $\BS$ through the
standard embedding. From Theorem \ref{thm:property_D_alpha_and_cyclic_groups}, we know that for every $0 <
\alpha < 1$, there exists a box space of $\BS$ with property $D_\alpha$, but what about arithmetic box spaces? 
We will prove that box spaces with $D_\alpha$, for $\alpha>\frac{1}{2}$, can be distinguished from arithmetic box spaces by coarse geometry. More precisely:

\begin{theorem}
  \label{thm:properties_arithmetic_box_spaces}
  For any $m \geq 2$, the following statements are true~:
  \begin{enumerate}
  \item if an arithmetic box space $\square_{(G_m(N_k))_k} G_m$
    has property $D_\alpha$, then $\alpha \leq \frac{1}{2}$,
    
  \item there exists an arithmetic box space with property $D_{1/2}$,
    
  \item there exists an arithmetic box space of
    $G_m$ without property $D_\alpha$ for any $\alpha \in ]0, 1/2]$.
  \end{enumerate}
\end{theorem}

If $\square_{(H_k)}G,\,\square_{(H_k')}G$ are two box spaces of the same residually finite group $G$, we say that $\square_{(H_k')}G$ covers $\square_{(H_k)}G$ if $H_k'\subset H_k$ for every $k>0$. In this case $\cay(G/H_k',S)$ is a Galois covering of $\cay(G/H_k,S)$. 

The following proposition bridges Theorem \ref{thm:properties_arithmetic_box_spaces} with part (2) of Theorem \ref{thm:property_D_alpha_and_cyclic_groups}.

\begin{proposition} Fix $\alpha<1$. Any arithmetic box space of $G_m$ is covered by some box space with $D_\alpha$.
\end{proposition}\label{covering}
\medskip
In addition, 
we study how property $D_\alpha$ for an arithmetic box space depends on the prime factors of the $N$'s in the sequence of congruence subgroups $(M_N)_N$. In fact, if we denote by $D'(P)$ the analytic density of the prime factors (see Section \ref{densityresults}), we prove the following.
\begin{theorem}
  \label{thm:density_results_intro}
  Let $\square_{(G_m(N_k))_k}G_m$ be an arithmetic box space, and let $P$ be the set of prime factors of
  the sequence $(N_k)_k$.
  \begin{enumerate}
  \item If $|P| < + \infty$, then $\square_{(G_m(N_k))_k}G_m$ has $D_{1/2}$~;
  \item If $D'(P) > 0$, then $\square_{(G_m(N_k))_k}G_m$ does not have $D_\alpha$, for every $\alpha>0$.
  \end{enumerate}
\end{theorem}

Of course the choice of the standard embedding raises a natural question: how do the congruence subgroups in
$BS(1,m)$ depend on the choice of the embedding $\rho$? Since congruence subgroups have finite index, this is
related to a form of the {\it congruence subgroup property} (CSP), which we now recall.

\begin{definition} Let $G$ be a subgroup of $GL_n(\Z[1/m])$. We say that $G$ has the CSP if every finite index subgroup in $G$ contains $G(N)$ for some $N>0$.
\end{definition}

Let $T_n(\Z[1/m])$ denote the group of upper triangular matrices in $GL_n(\Z[1/m])$. We shall
prove:

\begin{theorem}
  \label{thm:CSP-upper-triangular-intro}
  For every embedding $\rho:BS(1,m)\rightarrow T_n(\Z[1/m])$, the group $\rho(BS(1,m))$
  has the CSP.
\end{theorem}

This result can be reformulated as follows. Taking the congruence subgroups $\rho(G_m)(N)$
as a neighbourhood basis for the identity gives a topology on $\rho(G_m)$ (namely the congruence
topology relative to $\rho$) and the completion $\overline{\rho(G_m)}$ with respect to this topology is a profinite group called the congruence completion of $\rho(G_m)$. As $G_m$ is residually finite, it also embeds in its profinite completion $\widehat{G_m}$
which maps onto $\overline{\rho(G_m)}$. The CSP for $\rho(G_m)$ says that the map $\widehat{G_m}\rightarrow\overline{\rho(G_m)}$ is an isomorphism.

Another immediate consequence of CSP is that, if $\rho:G_m\rightarrow T_{n}(\Z[1/m])$ is an embedding, then any box space of $G_m\simeq\rho(G_m)$ is covered by some arithmetic box space of $\rho(G_m)$.

\medskip
When $m=p^\ell$ is a prime power, we also propose another approach to CSP for $\BS$: we view $BS(1,p^\ell)$ as a subgroup of the affine group $\mathbb{G}(\Q)$, and this subgroup is commensurable to the group $\mathbb{G}(\mathcal{O}_S)$ of $S$-integer points, with $S=\{p,\infty\}$. We prove:

\begin{proposition} Let $\mathbb{H}$ be a $\Q$-algebraic subgroup of $GL_n$, and let $\rho:\mathbb{G}\rightarrow\mathbb{H}$ be a $\Q$-isomorphism. Then $\rho(BS(1,p^\ell))$ has CSP.
\end{proposition}

As for the structure of the paper, we will start in Section \ref{sec:some-properties-multiplicative-order} by
recalling some well-known facts about elementary number theory, especially concerning the
multiplicative order of $m$ in $\Z/N\Z$. 

We then study property $D_\alpha$ for solvable Baumslag-Solitar groups in Section
\ref{sec:prop-d_alpha-solv}. We begin with some metric aspects, where we estimate the diameter of the
arithmetic box spaces of $BS(1,m)$, which is what we use to prove
Theorem \ref{thm:properties_arithmetic_box_spaces} and Proposition \ref{covering} (see Theorem \ref{thm:theorem-principal-1}). To conclude
Section \ref{sec:prop-d_alpha-solv}, we investigate the role of prime numbers in property $D_\alpha$ and we
prove Theorem \ref{thm:density_results_intro} (see Theorem \ref{thm:density_results}).

Finally, in Section \ref{sec:csp-bs1-m}, we prove Theorem \ref{thm:CSP-upper-triangular-intro}. 

\medskip
We will use Landau's notations $O,o,\Omega,\Theta$.


\section{Some elementary number theory}
\label{sec:some-properties-multiplicative-order}

We gather some technical lemmas. First we recall some facts about greatest common divisors and lowest common multiples, which will be
respectively denoted by $\gcd$ and $\lcm$.

\begin{proposition}
  \label{prop:properties_gcd_lcm}
  Let $\mu, a_1, \ldots , a_n \in \N^\ast$,
  \begin{enumerate}
  \item $\gcd(\mu \cdot a_1, \mu \cdot a_2) = \mu \cdot  \gcd(a_1, a_2)$,
  \item $\gcd(a_1, \ldots , a_n) = \gcd(\gcd(a_1, \ldots , a_{n-1}), a_n)$,
  \item $\gcd(a_1, \gcd(a_2, a_3)) = \gcd(\gcd(a_1,a_2), a_3)$,
  \item\label{item:4} $\gcd(a_1, a_2) \cdot \lcm(a_1, a_2) = a_1 \cdot a_2$,
  \item $\lcm(a_1, ..., a_n) = \lcm(\lcm(a_1, \ldots , a_{n-1}), a_n)$.
  \hfill$\square$
  \end{enumerate}
\end{proposition}

The following lemma generalises Proposition \ref{prop:properties_gcd_lcm}.(\ref{item:4}).

\begin{lemma}
  \label{lem:lcm=prod_divided_by_gcd}
  Let $a_1, \ldots , a_n \in \N$, then
  \begin{equation}
    \label{eq:lcm=prod_divided_by_gcd}
    \lcm(a_1, \ldots , a_n) = \frac{a_1 \cdots a_n}{\gcd(a_1 \cdots a_{n-1}, a_1 \cdots a_{n-2}a_n, \ldots ,
      a_2 \cdots a_{n})}.
  \end{equation}
\end{lemma}

\begin{proof}
  We use induction to show that Eq.~(\ref{eq:lcm=prod_divided_by_gcd}) is valid. If $n = 1$, the
  formula holds. Thus assume that the formula is true for $n \in \N$, and denote by $\Pi_n$ the set
  $\{a_1 \cdots a_{n-1}, a_1 \cdots a_{n-2}a_n, \ldots , a_2 \cdots a_{n}\}$. Using that 
  \[ \gcd \left( \lcm(a_1, \ldots , a_n), a_{n+1} \right) = \gcd \left( \frac{a_1 \cdots a_n}{\gcd(\Pi_n)},
      a_{n+1} \right), \]
  we obtain by a direct computation that
  \begin{align*}
    \lcm(a_1, \ldots , a_{n+1})
    &= \lcm(\lcm(a_1, \ldots , a_n), a_{n+1})\\
    &= \frac{\lcm(a_1, \ldots , a_n) \cdot a_{n+1}}{\gcd \left( \lcm(a_1, \ldots , a_n), a_{n+1} \right)}\\
    &= \frac{a_1 \cdots a_n \cdot a_{n+1}}{\gcd(\Pi_n) \gcd \left( \frac{a_1 \cdots a_n}{\gcd(\Pi_n)},
      a_{n+1} \right)}\\
    &= \frac{a_1 \cdots a_{n+1}}{\gcd(a_1 \cdots a_{n}, \gcd(\Pi_n) \cdot a_{n+1})}\\
    &= \frac{a_1 \cdots a_{n+1}}{\gcd(a_1 \cdots a_n, \gcd(a_1 a_3 \cdots a_n a_{n+1}, \ldots , a_2 \cdots
      a_{n} a_{n+1}))}\\
    &= \frac{a_1 \cdots a_{n+1}}{\gcd(a_1 \cdots a_n, a_1 a_3 \cdots a_n a_{n+1}, \ldots, a_2 \cdots
      a_{n} a_{n+1})}. \qedhere
  \end{align*}

\end{proof}

We now recall the ideal structure of the ring $\Z[1/m]$.
\begin{lemma}\label{prop:reduction_mod_N_is_nice}
Let $I$ be a proper subgroup in $\Z[1/m]$. TFAE:
\begin{enumerate}
    \item $I$ is an ideal in $\Z[1/m]$;
    \item There exists $N>1$ such that $I=N\Z[1/m]$.
\end{enumerate}
Moreover $\Z[1/m]/N\Z[1/m]\simeq\Z/N\Z$
\end{lemma}

\begin{proof} The non-trivial direction follows from general results about localisations of rings, viewing $\Z[1/m]$ as the localisation of $\Z$ with respect to powers of $m$: the map $I\rightarrow I\cap\Z$ provides a bijection between ideals of $\Z[1/m]$ and ideals $J$ in $\Z$ such that $m$ is not a zero-divisor in $\Z/J$ (see Proposition 2 in section 11.3 of \cite{Cohn}).
Finally observing that $\Z+N\Z[1/m]=\Z[1/m]$, we have by a classical isomorphism theorem:
\[ \Z[1/m] / N\Z[1/m] = ( \Z + N\Z[1/m] ) / N\Z[1/m] \simeq \Z / ( \Z \cap N\Z[1/m] ) = \Z / N\Z.\]
\end{proof}

Lemma \ref{prop:reduction_mod_N_is_nice} allows us to work with $\Z/N\Z$, which has a familiar ring
structure. We will write $\Z/N\Z^\times$ for the multiplicative group of $\Z/N\Z$. We denote by $\ord_m(N)$ the multiplicative order of $m$ in
$\Z/N\Z^\times$. We define the following function.

\begin{definition}\label{definitionnu}
   Let $m, N \in \N$ be such that $\gcd(m,N) = 1$. 
    Write $m^{\ord_m(N)} = \mu N + 1$ for some $\mu \in \N$ and let the function
    $\eta_N \colon \N^\ast \to \N$ be defined by
    \[ \eta_N(k) = \begin{cases} 1 & \text{if } k = 1\\ \frac{N^{k-1}}{\gcd(\mu, N)} & \text{if }k \geq
        2. \end{cases}  \]
\end{definition}

  \begin{lemma}
    \label{lem:orders}
        Let $m, N \in \N$ be such that $\gcd(m,N) = 1$. Write $m^{\ord_m(N)} = \mu N + 1$ for some $\mu \in \N$.  
      Then $\ord_m(N^2) = \ord_m(N) \cdot\frac{N}{\gcd(\mu, N)}$, and more generally
    \begin{equation}
      \label{eq:orders}
      \ord_m(N^k) = \ord_m(N) \cdot \eta_N(k), \quad \forall k \geq 1.
    \end{equation}
  \end{lemma}

  \begin{proof} The case $k = 1$ being obvious, let us consider $k = 2$, and set $\beta=\ord_m(N)$. We show that the smallest positive integer $\lambda$ that satisfies $m^\lambda \equiv 1 \pmod{N^2}$ is
    $\lambda = \beta \frac{N}{\gcd(\mu, N)}$. Note that $\beta \mid \lambda$~: indeed if
    $m^\lambda \equiv 1 \pmod{N^2}$, then $m^\lambda \equiv 1 \pmod N$ so that $\beta$ must divide $\lambda$
    (see \cite[Cor. 2 p. 79]{gallian2012contemporary}). Thus $\lambda = \beta \tilde{\lambda}$ for some
    $\tilde{\lambda} \in \N$ and we only have to show that $\tilde{\lambda} = \frac{N}{\gcd(\mu, N)}$.  We
    have that
    \begin{align*}
      (m^\beta)^{\tilde{\lambda}}
      &= (\mu N + 1)^{\tilde{\lambda}}\\
      &= \sum_{i = 0}^{\tilde{\lambda}} \binom{\tilde{\lambda}}{i} (\mu N)^i\\
      &\equiv 1 + \tilde{\lambda} \mu N \pmod{N^2}.
    \end{align*}
    The last line shows that the smallest $\tilde{\lambda}$ we can take to have $m^{\beta \tilde{\lambda}}
    \equiv 1 \pmod{N^2}$ must be $\tilde{\lambda} = \frac{N}{\gcd(\mu, N)}$, thus demonstrating that $\ord_m(N^2) =
    \beta \cdot \frac{N}{\gcd(\mu, N)}$.

    The same arguments can be applied to show that $\ord_m(N^k) = \beta \frac{N^{k-1}}{\gcd(\mu, N)}$ for $k
    \geq 2$.
  \end{proof}

  \begin{lemma}
    \label{lem:function_eta}
    Let $k, N \in \N^\ast$. Then $\eta_N(k) \geq N^{k-2}$.
  \end{lemma}

  \begin{proof}
    If $k = 1$, then $1 \geq N^{-1}$. If $k \geq 2$, observe that $\gcd(\mu, N) \leq N$ implies 
    \[ \frac{N^{k-1}}{\gcd(\mu, N)} \geq N^{k-2}. \]
  \end{proof}

  Denote by $\mathcal{P} \subset \N$ the set of prime numbers. The following lemma gives us a formula to
  compute the order of $m$ in $\Z/N\Z$ for any $N \in \N$. It is an immediate consequence of the Chinese
  remainder theorem.

\begin{lemma}
  \label{lem:order=lcm}
  For every $N \in \N$, which we write $N = p_1^{\beta_1}p_2^{\beta_2} \cdots p_n^{\beta_n}$ with $p_i \in
  \mathcal{P}$ and $\beta_i \in \N$ for every $i \in \{1, \ldots , n\}$, we have
  \begin{equation}
    \label{eq:order=lcm}
    \ord_m(N) = \ord_m\left(p_1^{\beta_1}p_2^{\beta_2} \cdots p_n^{\beta_n}\right) = \lcm\left(\ord_m\left(p_1^{\beta_1}\right),
    \ord_m\left(p_2^{\beta_2}\right), \ldots , \ord_m\left(p_n^{\beta_n}\right)\right).
  \end{equation}

\end{lemma}

\begin{lemma}
  \label{lem:lowerboundonratio}
  Let $P$ be a finite set of primes, not dividing $m$. There exists a constant $C(m,P)>0$ such that, for every integer $N$ with all prime factors in $P$, we have:
  $$\frac{\ord_m(N)}{N}\geq C(m,P).$$
\end{lemma}

\begin{proof}
  Write $N= p_1^{\beta_1} \ldots p_{k}^{\beta_{k}}$, with $p_i \in P$, all different, $\beta_i > 0$ and $\eta_{p_i}$ defined as in Definition \ref{definitionnu}. In addition, we define the set 
  \begin{align*}
      \Pi_k := \{ &\ord_m(p_1)\eta_{p_1}(\beta_1) \cdots \ord_m(p_{k - 1}) \eta_{p_{k -1}}(\beta_{k -
      1}),\\
      &\ord_m(p_1)\eta_{p_1}(\beta_1) \cdots \ord_m(p_{k - 2}) \eta_{p_{k -2}}(\beta_{k -
      2}) \ord_m(p_{k}) \eta_{p_{k}}(\beta_{k}), \\
      &\ldots\\
      &  \ord_m(p_2)\eta_{p_2}(\beta_2) \cdots \ord_m(p_{k}) \eta_{p_{k}}(\beta_{k})
  \} 
   \end{align*} 
which contains all possible products with $k-1$ factors, each one of $\ord_m(p_i)\eta_{p_i}(\beta_i)$, $i = 1, \ldots , k$.
  Using Lemmas~\ref{lem:lcm=prod_divided_by_gcd}, \ref{lem:orders}, and \ref{lem:order=lcm}, we obtain
  \begin{align*}
    \ord_m(N)
    &= \lcm\left(\ord_m\left(p_1^{\beta_1}\right),
      \ord_m\left(p_2^{\beta_2}\right), \ldots , \ord_m\left(p_{k}^{\beta_{k}}\right)\right)\\
    &= \frac{\ord_m(p_1) \eta_{p_1}(\beta_1) \cdots \ord_m(p_{k})
      \eta_{p_k}(\beta_{k})}{\gcd(\Pi_k)},
  \end{align*}
  thus
  \begin{equation}
    \label{eq:quotient}
    \frac{\ord_m(N)}{N} = \frac{\ord_m(p_1) \eta_{p_1}(\beta_1) \cdots \ord_m(p_{k})
      \eta_{p_k}(\beta_{k})}{N \cdot \gcd(\Pi_k)}.
  \end{equation}

Moreover, $\ord_m(p_i)\geq 1$ for every $i$, and using
  Lemma~\ref{lem:function_eta} on each $\eta_{p_i}(\beta_i)$, we obtain from Eq.~(\ref{eq:quotient}) 
  \begin{equation}
    \label{eq:part1-proof}
    \frac{\ord_m(N)}{N} \geq \frac{p_1^{\beta_1 - 2} \cdots p_{{k}}^{\beta_{{k}}-2}}{p_1^{\beta_1}
      \cdots p_{{k}}^{\beta_{{k}}} \cdot \gcd(\Pi_{k})} = \frac{1}{p_1^2 \ldots p_{{k}}^2 \cdot \gcd(\Pi_{k})} > 0.
  \end{equation}
  So we may take $C(m,P)$ as the minimum of the $\frac{1}{p_1^2 \ldots p_{k}^2 \cdot \gcd(\Pi_{k})}$'s taken over all subsets $\{p_1,...,p_k\}$ of $P$.

\end{proof}

The following material will be used in the proof of Proposition \ref{CSPfinal}
	\begin{definition}
	Let $N\in \NN$ and $N= \Pi_{i=1}^{r} p_i^{\beta_i} $ its decomposition in prime factors. The \textit{dominant prime} of $N$ is the factor $p_I$ in $P$ such that $p_I^{\beta_I}\ge p_i^{\beta_i}$ for $\forall i$. The factor $p_I^{\beta_I}$ is called the \textit{dominating factor}.
\end{definition}

	\begin{lemma}\label{lemmaoddordermodulo} Fix $s\in(\Z[1/m])^\times$, with $s>0$. There exist infinitely many integers $N>0$, coprime to $m$, such that $s$ has odd order in the multiplicative group $(\Z[1/m]/N\Z[1/m])^\times\simeq(\Z/N\Z)^\times$.
	\end{lemma}
	\begin{proof} If $s=1$, take any $N$ coprime with $m$. So we assume $s\neq 1$ and, replacing $s$ by $s^{-1}$ if necessary, we assume $s>1$.
	
		Let $P$ be the set of primes dividing $m$, set $r=|P|$. Write $s=\frac{a_1}{a_2}$ with $a_1>a_2>0$, coprime, and all their prime factors in $P$. For $k\in\NN$, set $a_1^k-a_2^k=  N_kq_k$, where $N_k$ is the maximal factor coprime to $m$. Observe that if a prime in $P$ divides $q_k$, it cannot simultaneously divide $a_1$ and $a_2$, since they are coprime.
		
		It is clear that, for odd $k$, $s$ will also have odd order in $(\Z[1/m]/(a_1^k-a_2^k)\Z[1/m])^\times$, hence also in $(\Z[1/m]/N_k\Z[1/m])^\times$. It is therefore enough to prove that the map $k\mapsto N_k$ takes infinitely many values on odd integers. This will follow immediately from the following.
		
		\medskip
		{\bf Claim:} There is an infinite family of odd integers $k$ such that $q_k \leq (a_1^{2r}-a_2^{2r})^r$.
		
\medskip
To prove the claim, take $k$ odd with $k>2r$. By the pigeonhole principle there are $i,j\in\{ 0,1\dots, r \}$ with $i<j$ such that the dominant primes in $q_{k-2i}$ and $q_{k-2j}$ are the same. We have that 
	\begin{align*}
		(s^{k-2i}-1)-(s^{k-2j}-1) &= s^{k-2j}(s^{2(j-i)}-1)\\
									&= \frac{a_1^{k-2j}}{a_2^{k-2i}}( a_1^{2(j-i)} -a_2^{2(j-i)}).
	\end{align*}			
		Write $p_i^{\beta_i}$ for the dominating factor of $q_{k-2i}$. Set $\beta =  \min\{\beta_i,\beta_j\}$ and $p=p_i$; say that $\beta=\beta_i$ (otherwise replace $i$ by $j$). We see that $p^\beta$ divides the numerator on the left, hence it divides the numerator on the right. We note that $p$ does not divide $a_1$, nor $a_2$, hence it must divide $a_1^{2(j-i)} -a_2^{2(j-i)}$, which is bounded by $(a_1^{2r}-a_2^{2r})$. So we get $p^\beta\leq (a_1^{2r}-a_2^{2r})$ and it follows that $q_{k-2i}\leq (a_1^{2r}-a_2^{2r})^r$.   	
	\end{proof}

\section{Property $D_\alpha$ for solvable Baumslag-Solitar groups}
\label{sec:prop-d_alpha-solv}

\subsection{Metric aspects of solvable Baumslag-Solitar groups}
\label{sec:gener-solv-baumsl}

We study the diameter of arithmetic box spaces of $\BS$ according to Eq.~(\ref{eq:property_D_alpha}). In this section, we will always assume that $\gcd(m, N) = 1$. We
recall that every element of $\BS$ ($m > 1$) admits a unique normal form of the type $t^{-i}a^\ell t^j$ with
$i, j \geq 0, \ell \in \Z$ and $\ell$ can be a multiple of $m$ only if either $i$ or $j$ is zero. Indeed, one
can rewrite $ta$ as $a^mt$, $ta^{-1}$ as $a^{-m}t$, $at^{-1}$ as $t^{-1}a^m$, and $a^{-1}t^{-1}$ as
$t^{-1}a^{-m}$ and the result follows.

\medskip
The normal form of a word is usually not the geodesic form, and we want to estimate how well the
normal form approximates the geodesic form.

\begin{proposition}[{\cite[Prop. 2.1]{burillo2015metric}}]
  \label{prop:burillo-elder}
  There exist constants $C_1, C_2, D_1, D_2 > 0$ such that for any $\omega = t^{-i} a^\ell t^j \in \BS$ with
  $\ell \neq 0$, we have 
  \[ C_1(i + j + \log |\ell|) - D_1 \leq \|\omega\| \leq C_2(i + j + \log |\ell|) + D_2 \]
  where $\| \cdot \|$ is the word metric with respect to $\{a^{\pm 1}, t^{\pm 1}\}$. Moreover we may take $C_2 = D_2
  = m$.
\end{proposition}

Let $A = \left(
\begin{smallmatrix}
  1 & 1 \\ 0 & 1
\end{smallmatrix}\right)
$, $T = \left(
\begin{smallmatrix}
  m & 0 \\ 0 & 1
\end{smallmatrix}\right) \in GL_2\left(\Z\left[ 1/m \right]\right)$, and denote by $G_m$ the
subgroup of $GL_2\left(\Z\left[ 1/m \right]\right)$ generated by $A$ and $T$. In the previous section,
we saw that $BS(1, m) \simeq G_m \leq GL_2\left(\Z\left[ 1/m \right]\right)$. 

\medskip
As mentioned before, $BS(1, m) \simeq G_m$ is a finitely generated, residually finite group that
surjects onto $\Z$ so that Theorem \ref{thm:property_D_alpha_and_cyclic_groups} applies and we know that for
every $0 < \alpha < 1$, there exists a box space of $G_m$ with property $D_\alpha$. However, we are interested in specific box spaces of $G_m$, namely the arithmetic box spaces, or in other words, box spaces
of the form $\square_{(G_m(N_k))_k} G_m$. To this end, we start by studying the quotients
$G_m/G_m(N)$, and then explore how the diameters evolve.

\begin{proposition}
  Let $N \in \N$ be such that $\gcd(m, N) = 1$. Then
  \begin{equation}
    \label{eq:quotient_Gamma}
    G_m/G_m(N) \simeq \Z/N\Z \rtimes_m \Z/\ord_m(N)\Z, \quad \text{ and } \quad \left| G_m/G_m(N)
    \right| = N \cdot \ord_m(N),
  \end{equation}
  where $\Z/\ord_m(N) \Z$ acts on $\Z/N\Z$ by multiplication by $m$.
\end{proposition}

\begin{proof}
  Consider reduction modulo $N$: 
\[
  \begin{array}{cccl}
    \phi \colon & G_m & \to & GL_2(\Z/N\Z) \\
    & w=\begin{pmatrix} m^k & x \\ 0 & 1 \end{pmatrix} & \mapsto & \begin{pmatrix}
      [m^k] & [x] \\ [0] & [1]
    \end{pmatrix}.
  \end{array}
\]
The image of $\phi$ is clearly isomorphic to $\Z/N\Z\rtimes_m \Z/\ord_m(N)\Z$, and of order $N\cdot\ord_m(N)$.
Moreover we have 
\[
 w= \begin{pmatrix}
    m^k & x \\ 0 & 1
  \end{pmatrix} \in G_m(N) \iff m^k \equiv 1 \pmod N \text{ and } x \in N\Z\left[\tfrac{1}{m}\right] \iff \phi(w)=1.
\]

So $\ker(\phi)=G_m(N)$ and the result follows from the first isomorphism theorem.
\end{proof}

\begin{example}
  Consider $BS(1,2)$ and $N = 5$. Then $\ord_2(5) = 4$, and $G_2/G_2(5) \simeq \Z/5\Z \rtimes_2
  \Z/4\Z$. The Cayley graph of the quotient is the graph drawn below, using $a = (1,0)$ and $t = (0,1)$ as generators. Note that one still needs to identify the bottom line with the upper line, and the line to the left with the line to the right.
    
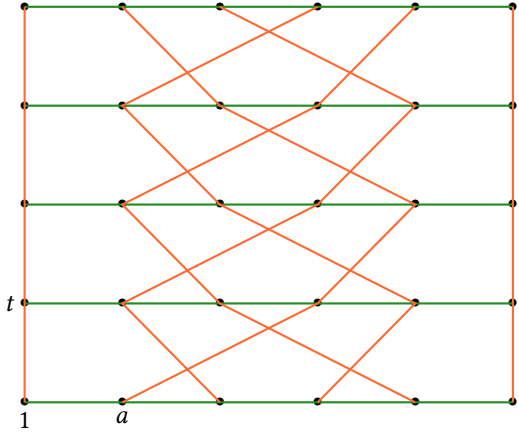
\begin{center}
    \begin{tikzpicture}[scale=0.65]
      \foreach \k in {0, 2,..., 10} {
        \foreach \l in {0, 2,..., 8} {
          \draw (\k, \l) node[scale=0.8] {$\bullet$};
        }
      }
      \draw (0,0) node[below] {$1$};
      \draw (2, 0) node[below] {$a$};
      \draw (0, 2) node[left] {$t$};
      \draw[thick, ForestGreen] (0,0) -- (2,0) -- (4,0) -- (6,0) -- (8, 0) -- (10,0);
      \draw[thick, ForestGreen] (0,2) -- (2,2) -- (4,2) -- (6,2) -- (8, 2) -- (10,2);
      \draw[thick, ForestGreen] (0,4) -- (2,4) -- (4,4) -- (6,4) -- (8, 4) -- (10,4);
      \draw[thick, ForestGreen] (0,6) -- (2,6) -- (4,6) -- (6,6) -- (8, 6) -- (10,6);
      \draw[thick, ForestGreen] (0,8) -- (2,8) -- (4,8) -- (6,8) -- (8, 8) -- (10,8);
      \draw[thick, OrangeRed!80] (0,0) -- (0, 2) -- (0, 4) -- (0, 6) -- (0,8);
      \draw[thick, OrangeRed!80] (2, 0) -- (6, 2) -- (8, 4) -- (4, 6) -- (2, 8);
      \draw[thick, OrangeRed!80] (4, 0) -- (2, 2) -- (6, 4) -- (8, 6) -- (4, 8);
      \draw[thick, OrangeRed!80] (8, 0) -- (4, 2) -- (2, 4) -- (6, 6) -- (8, 8);
      \draw[thick, OrangeRed!80] (6,0) -- (8,2) -- (4, 4) -- (2, 6) -- (6,8);
      \draw[thick, OrangeRed!80] (10,0) -- (10, 2) -- (10, 4) -- (10, 6) -- (10,8);
    \end{tikzpicture}
  \end{center}
\end{example}

\medskip
Thanks to the familiar structure of the quotient $G_m/G_m(N)$ and Proposition 2.1 from \cite{burillo2015metric}, we are able to estimate the diameter of arithmetic box spaces of $BS(1,m)$.

\begin{lemma}
  \label{lemma:diameter-bounds}
  Let $N \geq 2$. Then
  \begin{equation}
    \label{eq:estimation-diameter}
    \diam(\cay(G_m/ G_m(N))) = \Theta(\ord_m(N)).
  \end{equation}
  More precisely, there
  exists a constant $C_m > 0$ such that
  \begin{equation}
    \label{eq:diameter:upper-bound}
    \frac{1}{3}\cdot \ord_m(N)\leq\diam(\cay(G_m / G_m(N))) \leq C_m \cdot \ord_m(N).
  \end{equation}
\end{lemma}

\begin{proof}
  Let $m \geq 2$ and consider $\BS \simeq G_m \subset \mathsf{GL}_2(\Z[1/m])$.
  Recall that $G_m/G_m(N) \simeq \Z/N\Z \rtimes_m \Z/\ord_m(N)\Z$ so that
  $\diam(G_m/G_m(N)) = \diam(\Z/N\Z \rtimes_m \Z/\ord_m(N)\Z) \geq \diam(\Z/\ord_m(N)\Z)$.
  Since the Cayley graph of $\Z/\ord_m(N)\Z$ is a
  cycle, we can roughly estimate the diameter to obtain
    \begin{equation}
      \diam(G_m/G_m(N)) \geq \frac{1}{3} \ord_m(N).
    \end{equation}

    For the second inequality, let $([x], [k]) \in \Z/N\Z \rtimes_m \Z/\ord_m(N)\Z$ be an element realising the diameter. We rewrite
    $([x], [k])$ as $\big( \begin{smallmatrix} m^k & x \\ 0 & 1 \end{smallmatrix} \big)G_m(N)$. The induced
    metrics are always smaller in a quotient, thus
    \[ \left\| \begin{pmatrix} m^k & x \\ 0 & 1 \end{pmatrix} G_m(N) \right\|_{G_m/G_m(N)} \leq
      \left\| \begin{pmatrix} m^k & x \\ 0 & 1 \end{pmatrix} \right\|_{G_m}.\] Recall that any word
    $\omega \in \BS$ can be written in normal form as $\omega = t^{-i} a^{\ell} t^j$. In the quotient
    $G_m/G_m(N)$, the situation is even simpler, as by the semi-direct product stucture every word can be written as $A^\ell T^j$ with
    $0 \leq \ell < N$ and $0 \leq j < \ord_m(N)$.
    With $\ell = x$ and $j = k$, we identify $\big( \begin{smallmatrix} m^k & x \\ 0 & 1 \end{smallmatrix} \big)$ with the element $A^xT^k$ in normal form in $G_m$.
    If $x=0$ we get $T^k$ and
    \[ \|T^k\|_{G_m} = k < \ord_m(N). \] Assume that $x \neq 0$. From Proposition
    \ref{prop:burillo-elder}, we obtain
    \begin{equation}
      \label{eq:length-1}
      \left\| \begin{pmatrix} m^k & x \\ 0 & 1 \end{pmatrix} \right\|_{G_m} = \|A^x T^k\|_{G_m} \leq 2m(k +
      \log x + 1).
    \end{equation}
    Note that since $m^{\ord_m(N)} \geq N$ (equivalently 
    $\log(N) \leq \ord_m(N) \cdot \log(m)$) and moreover
    $\log x \leq \log(N)$, Eq.~(\ref{eq:length-1}) becomes
    \begin{equation}
      \label{eq:length-2}
      \left\| \begin{pmatrix} m^k & x \\ 0 & 1 \end{pmatrix} \right\|_{G_m}
      \leq 2m(2 + \log(m)) \ord_m(N).
    \end{equation}
    Setting $C_m := 2m(2 + \log(m))$, we obtain
    \begin{equation}
      \diam(\cay(G_m/G_m(N))) \leq C_m \cdot \ord_m(N).
    \end{equation} 

\end{proof}

\begin{proposition}\label{prop:charactD_alpha}
  An arithmetic box space $\square_{(G_m(N_k))_k}G_m$ has property $D_\alpha$ if and only if $$\ord_m(N_k)=\Omega(N_k^{\frac{\alpha}{1-\alpha}}).$$
\end{proposition}

\begin{proof} Using Lemma \ref{lemma:diameter-bounds}:
$$\square_{(G_m(N_k))_k}G_m\,\mbox{has}\,D_\alpha \iff \diam(G_m/G_m(N_k))=\Omega(|G_m/G_m(N_k)|^\alpha)$$
$$\iff \ord_m(N_k)=\Omega(N_k^\alpha\cdot \ord_m(N_k)^\alpha)\iff \ord_m(N_k)=\Omega(N_k^{\frac{\alpha}{1-\alpha}}).$$
  \end{proof}

We present here the main structure theorem for the arithmetic box spaces of $\BS$.

\begin{theorem}
  \label{thm:theorem-principal-1}
  For any $m \geq 2$, the following statements hold~:
  \begin{enumerate}
  \item \label{item:D_alpha_smaller_than_1_2} if an arithmetic box space $\square_{(G_m(N_k))_k}G_m$
    has property $D_\alpha$, then $\alpha \leq \frac{1}{2}$,
    
  \item \label{item:box_space_with_D_1_2}  there exists an arithmetic box space with property $D_{1/2}$,
    
  \item \label{item:box_space_without_D_alpha} there exists an arithmetic box space of
    $G_m$ without property $D_\alpha$ for any $\alpha \in ]0, 1/2]$.
    \item Fix $\alpha<1$. Every arithmetic box space of $G_m$ is covered by some box space with $D_\alpha$.
  \end{enumerate}
\end{theorem}

\begin{proof}
  \begin{enumerate}
  \item If $\square_{(G_m(N_k))_k}G_m$
    has property $D_\alpha$, using $N_k\geq\ord_m(N_k)$ and Proposition \ref{prop:charactD_alpha}, we get $N_k=\Omega(N_k^{\frac{\alpha}{1-\alpha}})$, which forces $\alpha\leq\frac{1}{2}$.

  \item Let $(N_k)_k \subset \N$ be the sequence defined by $N_k = (m^2-1)^k$. Clearly, $N_k \mid N_{k+1}$ for every $k$. We apply Lemma \ref{lem:orders} with $N=m^2-1$, so that $\ord_m(m^2-1)=2$ and $\mu=1$.
  We thus obtain:
    \begin{equation}
      \label{eq:order_D1_2}
      \ord_m(N_k) = 2 \cdot (m^2-1)^{k-1}, \forall k \geq 1.
    \end{equation}
    i.e. $\ord_m(N_k)=\Omega(N_k)$. By Proposition \ref{prop:charactD_alpha} the box space $\square_{(G_m(N_k))_k}G_m$ has property $D_{1/2}$.

  \item We consider the sequence $(N_k)_k$ defined by $N_k = m^{2^k}-1$ and prove that the arithmetic box space
    $\square_{(G_m(N_k))_k}G_m$ does not have property $D_\alpha$ for any $\alpha \in ]0,1/2]$. It is
    straightforward that 
    $N_k \mid N_{k+1}$ for every $k$, and $\ord_m(N_k) = 2^k$. 
    We have
    \[ \lim_{k \to \infty} \frac{\ord_m(N_k)}{N_k^\frac{\alpha}{1-\alpha}} = \lim_{k \to \infty}
      \frac{2^k}{(m^{2^k}-1)^{\frac{\alpha}{1-\alpha}}} = 0, \]
    i.e. $\ord_m(N_k)=o(N_k^{\frac{\alpha}{1-\alpha}})$. By Proposition \ref{prop:charactD_alpha} this shows that the arithmetic box space
    $\square_{(G_m(m^{2^k}-1))_k}G_m$ does not have property $D_\alpha$ for any $\alpha \in ]0,1/2]$.
    
    \item We adapt the proof of Theorem 3 in \cite{khukhro2017expanders}. Pick an integer $D>0$ with $\frac{D}{D+1}\ge\alpha$. Let $\square_{G_m(N_k)}G_m$ be any arithmetic box space of $G_m$. Define 
    $$n_k=\ord_m(N_k)\cdot N_k^D$$
    and
    $$M_k=\left\{\left(\begin{array}{cc}m^{\ell n_k}  & r \\0 & 1\end{array}\right):\ell\in\Z,\;r\in N_k\Z[\tfrac{1}{m}]\right\}.$$
    It is readily checked that $M_k$ is a subgroup and, because $\ord_m(N_k)|n_k$, that $M_k$ is normal in $G_m$ and is contained in $G(N_k)$. As $n_k|n_{k+1}$, we have that $\square_{(M_k)}G_m$ is a box space which covers $\square_{(G_m(N_k))}G_m$. 
    
    It remains to check that $\square_{(M_k)}G_m$ has property $D_\alpha$. But $G/M_k$ maps onto the cyclic group $\Z/n_k\Z$ so we have
    $$\diam(G/M_k)\ge\diam(\Z/n_k\Z)\ge \frac{n_k}{3}.$$
    On the other hand
    $$|G/M_k|^\alpha =n_k^\alpha N_k^\alpha=\ord_m(N_k)^\alpha N_k^{\alpha(D+1)}\le
    \ord_m(N_k)N_k^D=n_k.$$
    This concludes the proof.
    
  \end{enumerate}

  \end{proof}

\subsection{Density results}\label{densityresults}
  
A natural question after encountering the constructions of Theorem \ref{thm:theorem-principal-1}.\ref{item:box_space_with_D_1_2} and
\ref{thm:theorem-principal-1}.\ref{item:box_space_without_D_alpha} is ``how many arithmetic box spaces of $\BS$ have $D_{1/2}$''? In the
following paragraphs, we give a partial answer to this question.

Let $(N_k)_k \subset \N$ be such that $N_k \mid N_{k+1}$ for every $k > 0$, and denote by $P_k$ the set of
prime factors of $N_k$. Moreover, we define the set of prime factors of the sequence $(N_k)_k$ by
\begin{equation}
  \label{eq:prime_factors_sequence}
  P := \bigcup_{k=1}^{+\infty} P_k.
\end{equation}

Before stating our main result from this section, we need to introduce some definitions about the density of
prime numbers. We follow Powell~\cite{powell1980primitive} for the terminology.

\begin{definition}
  Let $P \subset \mathcal{P}$ be a subset of the prime numbers. The \textit{natural primitive density} of $P$ is
  (if the limit exists) 
  \[ d'(P) := \lim_{N \to \infty} \frac{\left| \{p \leq N \mid p \in P\}\right|}{\left| \{p \leq N \mid p \in
        \mathcal{P}\}\right|}. \]
  The \textit{analytic primitive density} of $P$ is (if the limit exists) 
  \[ D'(P) = \lim_{s \to 1^+} \frac{\sum_{p \in P}^{} \frac{1}{p^s}}{\sum_{p \in \mathcal{P}}^{} \frac{1}{p^s}}. \]
\end{definition}

If $P$ is finite then $d'(P)=D'(P)=0$. Suppose now that $D'(P) > 0$. In this case, we see that $\sum_{p \in P}^{}
  \frac{1}{p} = + \infty$, otherwise $D'(P)$ would be equal to $0$. Observe that 
  \[ \prod_{p \in P} \left(1 - \frac{1}{p}\right) = 0 \iff \sum_{p \in P}^{} \ln\left( 1 - \frac{1}{p}
    \right) = - \infty. \]
  But using that $\ln(1 + x) \leq x$ for $x > -1$
  \[ \sum_{p \in P}^{} \ln \left( 1 - \frac{1}{p}\right) \leq - \sum_{p \in P}^{} \frac{1}{p} = - \infty.\]
  Therefore, we obtain
  \begin{equation}
    \label{eq:prod_goes_to_0}
    \prod_{p \in P} \left(1 - \frac{1}{p}\right) = 0
  \end{equation}
  if $D'(P)>0$.

  \begin{theorem}
    \label{thm:density_results}
  Let $\square_{(G_m(N_k))_k}G_m$ be an arithmetic box space, and let $P$ be the set of prime factors of
  the sequence $(N_k)_k$.
  \begin{enumerate}
  \item If $|P| < + \infty$, then $\square_{(G_m(N_k))_k}G_m$ has $D_{1/2}$~;
  \item If $D'(P) > 0$, then $\square_{(G_m(N_k))_k}G_m$ does not have $D_\alpha$ for every $\alpha>0$.
  \end{enumerate}
\end{theorem}

\begin{proof}
In view of Proposition \ref{prop:charactD_alpha}, we must study the asymptotics of the quotient
  $\frac{\ord_m(N_k)}{N_k}$. 
  \begin{enumerate} 
    \item By Lemma \ref{lem:lowerboundonratio}, there exists a constant $C(m,P)$ such that $\frac{\ord_m(N_k)}{N_k}\geq C(m,P)$, i.e. $\ord_m(N_k)=\Omega(N_k)$. Proposition \ref{prop:charactD_alpha} applies to show that $\square_{(G_m(N_k))_k}G_m$ has $D_{1/2}$.
  
\item Assume now $D'(P)>0$, pick $N=p_1^{\beta_1}...p_k^{\beta_k}$ with $p_1,...,p_k\in P$. We have $\ord_m(N)\leq \phi(N)=|(\Z/N\Z)^\times|$, where $\phi$ denotes Euler's totient function. Then
$$\frac{\ord_m(N)}{N}\leq \prod_{i=1}^k\frac{\phi(p_i^{\beta_i})}{p_i^{\beta_i}}=\prod_{i=1}^k\left(1-\frac{1}{p_i}\right).$$
In view of Eq.~(\ref{eq:prod_goes_to_0}) we then get $\ord_m(N_k)=o(N_k)$, which proves the second part of the theorem thanks to Proposition~\ref{prop:charactD_alpha}.
  
\end{enumerate}
\end{proof}

We state here a few open questions related to the previous theorem. 
\begin{itemize}
    \item If we assume that $d'(P) > 0$, does the
associated arithmetic box space have $D_\alpha$ for some $\alpha \in ]0, 1/2[$ or not?
\item What happens in the case $|P| = + \infty$ and $D'(P) = 0$?
\item Given $\alpha \in ]0,1/2]$, can we create an arithmetic box space with
exactly $D_\alpha$?
\end{itemize}


\section{CSP for $BS(1,m)$}
\label{sec:csp-bs1-m}

For a subring $A$ of $\Q$, we define three subgroups of $GL_n(A)$:
\begin{itemize}
    \item $T_n(A)$, the subgroup of upper triangular matrices;
    \item $U_n(A)$, the unipotent subgroup, i.e. the subgroup of $T_n(A)$ consisting of matrices with 1's down the diagonal;
    \item $D_n(A)$, the subgroup of diagonal matrices.
\end{itemize}
The map $\Delta:T_n(A)\rightarrow D_n(A)$ taking a matrix to its diagonal, is a surjective group homomorphism with kernel $U_n(A)$. 

We will also need the set $N_n(A)$ of upper triangular nilpotent matrices, i.e. upper triangular matrices with 0's down the diagonal. Note that $T_n(A)={\bf 1}_n + N_n(A)$.

\subsection{Representations of $A$ into $U_n(A)$}

We recall Chernikov's theorem (see Theorem 4.10 in \cite{Warfield}): if $N$ is a torsion-free nilpotent group,
for every $k\geq 1$ the map $N\rightarrow N:x\mapsto x^k$ is injective.

The first lemma discusses one-parameter subgroups in $U_n(A)$.

\begin{lemma}\label{oneparameter} For every $g\in U_n(A)$, there exists a unique homomorphism $\alpha:A\rightarrow U_n(A)$ such that $\alpha(1)=g$.
\end{lemma}

\begin{proof} The existence follows from Cor. 10.25 in \cite{Warfield}: if $g={\bf 1}_n + X$, with $X\in N_n(A)$, then for $r\in A$:
$$\alpha(r)=({\bf 1}_n+X)^r=\sum_{k=0}^\infty \binom{r}{k} X^k={\bf 1}_n+ rX+\frac{r(r-1)}{2}X^2+...$$
(note that the sum is finite as $X^n=0$). For the uniqueness, let $\beta$ be another homomorphism with
$\beta(1)=g$; for $r\in A$, write $r=\frac{a}{b}$ with $a,b\in\Z, b>0$ and $a,b$ coprime. Then
$$\alpha(r)^b=\alpha(a)=g^a=\beta(a)=\beta(r)^b,$$
so $\alpha(r)=\beta(r)$ by Chernikov's theorem.
\end{proof}

\begin{definition} For a subgroup $H$ of $U_n(A)$, the \textit{isolator} of $H$ is:
$$I(H)=\{g\in U_n(A): g^k\in H\, \mbox{for some} \,k\geq 1\}.$$
\end{definition}

By Theorem 3.25 in \cite{Warfield}, $I(H)$ is a subgroup of $U_n(A)$. Clearly $H\subset I(H)$.

\begin{lemma}\label{finiteexponent} For $m\geq 2$, let $\alpha:\Z[1/m]\rightarrow U_n(\Z[1/m])$ be an injective homomorphism. Then $I(\alpha(\Z[1/m]))$ is abelian and the exponent of the group $I(\alpha(\Z[1/m]))/\alpha(\Z[1/m])$ is finite.
\end{lemma}

\begin{proof} The proof is in three steps.
\begin{enumerate}
\item By Lemma \ref{oneparameter}, there exists a unique homomorphism $\tilde{\alpha}:\Q\rightarrow U_n(\Q)$
  that extends $\alpha$. We show that $I(\alpha(\Z[1/m])\subset \tilde{\alpha}(\Q)$, from which the first
  statement will follow. For $g\in I(\alpha(\Z[1/m]))$, there exists $k\geq 1$ and $r\in \Z[1/m]$ such that
  $g^k=\alpha(r)$. Then
$$\tilde{\alpha}\left(\frac{r}{k}\right)^k=\tilde{\alpha}(r)=\alpha(r)=g^k,$$
hence $g=\tilde{\alpha}(\frac{r}{k})$ by Chernikov's theorem.

\item Let us show that there exists some non-zero $x_0\in \Z[1/m]$ such that, for
  $g\in I(\alpha(\Z[1/m])), k\geq 1, r\in \Z[1/m]$:
$$g^k=\alpha(r)\Longrightarrow \frac{r}{k}x_0\in \Z[1/m].$$
Write $\alpha(1)={\bf 1}_n+X$, with $X\in N_n(\Z[1/m])$, as in Lemma \ref{oneparameter}. Note that $X\neq 0$
as $\alpha$ is injective. Then
$$g=\tilde{\alpha}\left(\frac{r}{k}\right)=({\bf 1}_n+X)^{\frac{r}{k}}={\bf 1}_n + \frac{r}{k}X+...$$
For $1\leq i<n$ and an upper triangular matrix $Y$ of size $n\times n$, we denote by $Y_{(i)}$ the $i$-th
parallel to the diagonal (moving upwards from the diagonal). Let $i$ be the smallest index such that
$X_{(i)}\neq 0$. Since $(X^k)_{(i)}=0$ for $k\geq 2$, we have $g_{(i)}=\frac{r}{k}X_{(i)}$. Let $x_0$ be any
non-zero coefficient of $X_{(i)}$. Since $g\in U_n(\Z[1/m])$, we have $\frac{r}{k}x_0\in \Z[1/m]$ as desired.

\item Let $\pi(m)$ be the set of primes dividing $m$. An integer is a \textit{$\pi(m)$-number} if all its
  prime divisors are in $\pi(m)$. Write $x_0=\frac{b}{t}$, with $t$ a $\pi(m)$-number, and $b\in\Z$ is coprime
  to $t$. For $g\in I(\alpha(\Z[1/m]))$, with $g^k=\alpha(r)$ as above, write $\frac{r}{k}=\frac{a}{s\ell}$,
  where $a,s,\ell$ are pairwise coprime, $s$ is a $\pi(m)$-number and $\ell$ is coprime with $m$. By the
  previous step $\frac{r}{k}x_0=\frac{ab}{st\ell}\in \Z[1/m]$. Since $a$ and $\ell$ are coprime, this may
  happen only if $\ell$ divides $b$. Finally
$$g^b=\tilde{\alpha} \left( \frac{r}{k} \right)^b = \tilde{\alpha} \left( \frac{br}{k} \right) =
\tilde{\alpha} \left( \frac{ab}{s\ell} \right).$$ But $\frac{ab}{s\ell}\in \Z[1/m]$ as $\ell$ divides
$b$. This implies that $g^b\in\alpha(\Z[1/m])$, hence the exponent of $I(\alpha(\Z[1/m]))/\alpha(\Z[1/m])$
divides $b$.
\end{enumerate}
\end{proof}

\subsection{Special representations of $BS(1,m)$}

For $m\geq 2$, set $G_m=BS(1,m)=\Z[1/m]\rtimes\Z$. We will write $A_m$ for $\Z[1/m]$ when viewed as a normal subgroup of $G_m$.

\begin{definition}
  A \textit{ special representation} of $G_m$ is an injective homomorphism $\rho:G_m\rightarrow T_n(\Z[1/m])$
  such that $\rho(A_m)\subset U_n(\Z[1/m])$.
\end{definition}

We note that the standard embedding $G_m\rightarrow T_2(\Z[1/m])$, is a special representation.

\begin{lemma}\label{specialreps}
  \begin{enumerate}
  \item If $\rho$ is a special representation, then $\rho^{-1}(U_n(\Z[1/m]))=A_m$.
  \item If $m$ is even, then any injective homomorphism $\rho:G_m\rightarrow T_n(\Z[1/m])$ is a special
    representation.
  \end{enumerate}
\end{lemma}

\begin{proof} 
  We work with the presentation $G_m=\langle a,t|tat^{-1}=a^m\rangle$, observing that the normal subgroup
  $A_m$ coincides with the normal subgroup generated by $a$.
  \begin{enumerate}
  \item Suppose by contradiction that $A_m$ is strictly contained in $\rho^{-1}(U_n(\Z[1/m]))$. Then there
    exists $k>0$ such that $\rho(t^k)\in U_n(\Z[1/m])$. Consider the subgroup $H$ of $G_m$ generated by
    $A_m\cup\{t^k\}$, so that $\rho(H)\subset U_n(\Z[1/m])$. Since $\rho$ is injective, we see that $H$ is
    nilpotent. As $H$ also has finite index in $G_m$, we deduce that $G_m$ is virtually nilpotent, which is a
    contradiction.
  \item Suppose that $m$ is even and $\rho:G_m\rightarrow T_n(\Z[1/m])$ is an injective homomorphism. It is
    enough to see that $\rho(a)$ belongs to $U_n(\Z[1/m])=\ker(\Delta)$. But we have
    $$\rho(a^{m-1})=\rho([a,t])\in [T_n(\Z[1/m]),T_n(\Z[1/m])]\subset \ker(\Delta).$$
    Now the image of $\Delta$, namely $D_n(\Z[1/m])\simeq (\Z[1/m]^\times)^n$, contains only 2-torsion; since
    $\Delta(\rho(a))^{m-1}={\bf 1}_n$ and $m-1$ is odd, we have $\Delta(\rho(a))={\bf 1}_n$ as desired.
  \end{enumerate}
\end{proof}

We now head towards CSP for special representations of $BS(1,m)$. The next lemma is proved exactly as Lemma 4
in Formanek \cite{Form}, using the same ingredient, namely a number-theoretical result by Chevalley
\cite{Cheval}.

\begin{lemma}\label{Forma} Let $R$ be a subring of a number field with $R^\times$ finitely generated, let $G$
  be a subgroup of $D_n(R)$. There exists a function $\phi:\N\rightarrow\N$ such that, if $g\in G$ satisfies
  $g\equiv{\bf 1_n} \mod\phi(r)$, then $g$ is an $r$-th power in $G$.  \hfill$\square$
\end{lemma}

In Theorem 5 of \cite{Form}, Formanek proved that, if $R$ is the ring of integers of a number field and $G$ is
a subgroup of $T_n(R)$, then $G$ has CSP. The proof of the next result is inspired by Formanek's proof.

\begin{proposition}
  \label{propalain} 
  Let $\rho:G_m\rightarrow T_n(\Z[1/m])$ be a special representation of $BS(1,m)$. Then $\rho(G_m)$ has CSP.
\end{proposition}

\begin{proof} 
  Let $N\triangleleft G_m$ be a normal subgroup of finite index $r$. Denote by $G_m^r$ the subgroup generated
  by $r$-th powers in $G_m$. Then $G_m^r\subset N$ by Lagrange's theorem. Let $b$ be the exponent of
  $I(\rho(A_m))/\rho(A_m)$ (which is finite by Lemma \ref{finiteexponent}), and let $e$ be the exponent of the
  finite group $T_n(\Z[1/m]/(br)^2\Z[1/m])$. Define then
  $$M=(br)^2\phi(re),$$
  where $\phi$ comes from Lemma \ref{Forma} applied to $\Delta(\rho(G_m))$. We will show that if
  $\rho(g)\equiv {\bf 1}_n \mod M$, then $g\in G_n^r$, so that $\rho(N)$ contains the congruence subgroup
  $\rho(G_m)(M)$.

  If $\rho(g)\equiv {\bf 1}_n \mod M$, in particular $\Delta(\rho(g))\equiv {\bf 1}_n \mod\phi(er)$. By Lemma
  \ref{Forma}, the matrix $\Delta(\rho(g))$ is an $(er)$-th power in $\Delta(\rho(G_m))$, i.e. there exists
  $z\in G_m$ such that
  $$\Delta(\rho(g))=\Delta(\rho(z^{er}))=\Delta(\rho((z^e)^r)).$$
  But $\rho(z^e)\equiv {\bf 1}_n \mod (br)^2$, by definition of $e$, so also
  $\rho(z^{er})\equiv {\bf 1}_n \mod (br)^2$. On the other hand by definition of $M$ we have
  $\rho(g)\equiv{\bf 1}_n \mod (br)^2$, so $\rho(g^{-1}z^{er})\equiv {\bf 1}_n \mod (br)^2$. Since
  $g^{-1}z^{er}\in \ker(\Delta)=U_n(\Z[1/m])$, Lemma 1 of \cite{Form}\footnote{Formanek states this lemma for
    the ring
    of integers of a number field, but the proof only uses that the ring is commutative with characteristic
    0.} applies to guarantee that $g^{-1}z^{er}$ is a $(br)$-th power in $U_n(\Z[1/m])$, so we find
  $h\in U_n(\Z[1/m])$ such that
  $$\rho(g^{-1}z^{er})=h^{br}.$$
  Hence $g^{-1}z^{er}\in\rho^{-1}(U_n(\Z[1/m]))=A_m$ (the equality follows from the first part of Lemma
  \ref{specialreps}). This means that $h\in I(\rho(A_m))$, so by definition of $b$ we have $h^b\in\rho(A_m)$,
  say $h^b=\rho(y)$ for $y\in A_m$. Then $h^{br}=\rho(y^r)$ and $\rho(g^{-1}z^{er})=\rho(y^r)$. As $\rho$ is
  injective $g^{-1}z^{er}=y^r$, i.e $g=(z^e)^ry^{-r}$, so $g\in G_m^r$.
\end{proof}

\subsection{From special representations to all injective representations in $T_n(\Z[1/m])$}

In this section we show that for any injective representation $\rho$ of $G_m$ into $T_n(\Z[1/m])$, the group
$\rho(G_m)$ has CSP. The main idea is to pass to subrepresentations which are special and extract information
about $\rho$.

\begin{lemma}
  \label{lemmasubgroupcontaining}
  Let $G$ be a subgroup of $GL_n(\Z[1/m])$.
	
  \begin{enumerate}
  \item Let $H$ be a finite index subgroup of $G$, where $H$ has CSP. If $H$ contains a congruence subgroup of
    $G$, then $G$ also has CSP.
  \item For $1\le i \le k$, let $H_i$ be a finite index subgroup of $G$. Suppose $G= \cup_{i=1}^k H_i$, then
    if all $H_i$ have CSP, $G$ also has CSP.
  \end{enumerate}
\end{lemma}

\begin{proof}
  \begin{enumerate}
  \item Let $N$ be a finite index subgroup of $G$ , then $N\cap H$ is a finite index subgroup of $H$ and thus
    contains some congruence subgroup of the form $H(M)$. By assumption there is some $\tilde M$ such that
    $G(\tilde{M})\subset H$, hence we have that
		\[  G(M\tilde{M})\subset G(M)\cap G(\tilde{M})\subset H(M)\subset
		 N\cap H \subset N.    \]
		 
  \item Suppose by contradiction that $G$ does not have CSP, then there is some finite index subgroup $N$
       which does not contain any congruence subgroup. However each $N\cap H_i$ contains a congruence subgroup
       of the form $H(M_i)$. Set $\overline{M}= \Pi_i M_i$, then there is some
       $g\in G(\overline{M})\backslash N $. Note that this $g$ is necessarily in one of the $H_i$'s. So it is
       necessarily in $H(M_i)$ since $H(\overline{M})\subset H(M_i)$. However it is not in $N\cap H_i$: a
       contradiction.
  \end{enumerate}
\end{proof}

\begin{lemma}
  \label{positivetlemma}
  Let $\rho: G_m\rightarrow T_n(\Z[1/m])$ be a monomorphism. Suppose $\Delta(\rho(a))\ne {\bf 1}_n$ and the
  diagonal of $\rho (t)$ is strictly positive, then $\rho(G_m) $ has CSP.
\end{lemma}
\begin{proof}
  Note that the subgroup $B=\langle a^2, t\rangle$ is also a Baumslag-Solitar group isomorphic to $\BS$. The
  restriction $\rho\vert_{B}$ is a special representation of $\BS$; indeed by the proof of Lemma
  \ref{specialreps}, we have $\Delta(\rho(a))^{m-1}={\bf 1}_n$, so $\Delta(\rho(a))$ has finite order in
  $D_n(\Z[\frac{1}{m}])$; but the torsion subgroup of $D_m(\Z[\frac{1}{m}])$ is $\{\pm 1\}^n$, so
  $\Delta(\rho(a^2))={\bf 1}_n$. By Proposition \ref{propalain}, the group $\rho(B)$ has the CSP.
		
  We show there is some congruence subgroup of $\rho(G_m)$ that is fully contained inside
  $\rho(B)$. 
  Since $\Delta(\rho(a))\ne {\bf 1}_n$, there is some $-1$ on the diagonal, suppose it appears at the
  $i^{\text{th}}$ position. Let $s$ be the element of $(\Z[1/m])^\times$ in the $i^{\text{th}}$ position in
  $t$. By Lemma \ref{lemmaoddordermodulo}, we find $M>1$, coprime with $m$, such that $s$ has odd order in
  $(\Z[1/m]/M\Z[1/m])^\times$.
		
  We note that if $\omega$ is a word with letters in $\{a,t\}$ representing an element of $G_m\setminus B$,
  then it contains an odd number of $a$'s. Hence, the element on the $i^{\text{th}}$ position on the diagonal
  of $\rho(\omega)$ is of the form $-s^r$ for some $r\in \ZZ$. If we suppose that
  $\rho(\omega) \in \rho(G_m)(M)$, then $s^r = -1 \mod M$.  This is impossible since $s$ has odd order in
  $(\Z[\frac{1}{m}]/M\Z[1/m])^\times$ and thus cannot have a power equal to $-1$. So $\rho(\omega) $ cannot be
  in the corresponding congruence subgroup, i.e. $\rho(G_m)(M)\subset\rho(B)$. By Lemma
  \ref{lemmasubgroupcontaining}, the result follows.
\end{proof}

We are ready to complement Proposition \ref{propalain}.

\begin{proposition}
  \label{CSPfinal}
  For any injective homomorphism $\rho: G_m\rightarrow T_n(\Z[1/m])$, the group $\rho(G_m)$ has CSP.
\end{proposition}

\begin{proof}
  Since we already know the result for $m$ even (by Proposition \ref{propalain} and Lemma \ref{specialreps}),
  we may and will assume that $m$ is odd.

  The group $G_m$ has three subgroups $B_1=\langle a^2, t\rangle$, $B_2=\langle a^2, at\rangle$ and
  $B_3=\langle a,t^2\rangle$, which are all isomorphic to $\BS$ or $\BStwo$. We show that they cover
  $G_m$. Any element $\omega\in G_m$ can be described by a word of the form $t^{-k} a^r t^\ell$, with
  $k,\ell\in \NN$ and $r\in \ZZ$. Replacing $\omega$ by $\omega^{-1}$ if necessary, we may assume
  $k\leq \ell$.
  \begin{itemize}
  \item If $r$ is even, then $\omega \in B_1$;
  \item If $k$ and $l$ are both even or odd, then $\omega \in B_3$; indeed
    $\omega=t^{-2k}(t^k a^rt^{-k})t^{k+\ell}=t^{-2k}a^{rm^k}t^{k+\ell}\in B_3$ as $k+\ell$ is even.
  \item If $r$ is odd and $k,\ell$ do not have the same parity, then $\omega \in B_2$. Indeed
    \[t^{-k} a^r t^\ell = (t^{-1} a^{-1})^k a^{\sum_{i=0}^{k-1} m^i} a^r a^{-\sum_{i=0}^{\ell-1} m^i}
      (at)^l=(at)^{-k} a^{r -\sum_{i=k}^{\ell-1} m^i} (at)^\ell.\]
    Since $r,\ell-k$ and $m$ are odd, the exponent of $a$ in the latter expression is even, so $\omega\in B_2$.
		
  \end{itemize}
	
  Note that if $\Delta(\rho(a))= {\bf 1}_n$, then it is a special representation and Proposition
  \ref{propalain} applies. If it is not then the restrictions of $\rho$ to $B_1$ and $B_2$ are special
  representations, and the restriction of $\rho$ to $B_3$ satisfies the assumptions of Lemma
  \ref{positivetlemma}. Hence the $\rho(B_i)$'s all have CSP. Now simply applying Lemma
  \ref{lemmasubgroupcontaining} finishes the argument.
\end{proof}

\subsection{An alternative approach to CSP for $BS(1,p^\ell)$ ($p$ prime)}
We use the language of $\Q$-algebraic groups. Denoting by $\mathbb{G}_a$ (resp. $\mathbb{G}_m$) the additive
group (resp. the multiplicative group), we set $\mathbb{G}=\mathbb{G}_a\rtimes\mathbb{G}_m$, the affine group
viewed as a subgroup of $GL_2$ via the standard embedding.

Let $p_1,...,p_r$ be distinct primes, and let $S=\{p_1,...,p_r,\infty\}$ be viewed as a set of places of
$\Q$. Then the ring $\mathcal{O}_S$ of $S$-integers in $\Q$ is precisely
$$\mathcal{O}_S=\Z[1/p_1,...,1/p_r],$$
so that $\mathcal{O}_S^\times=\{\pm p_1^{k_1}...p_r^{k_r}:k_1,...,k_r\in\Z\}$ and
$$\mathbb{G}(\mathcal{O}_S)=\left\{\left(\begin{array}{cc}\pm p_1^{k_1}...p_r^{k_r} & a \\0 & 1\end{array}\right):k_1,...k_r\in\Z,\;a\in \Z[1/p_1,...,1/p_r]\right\}.$$
It is known that $\mathbb{G}(\mathcal{O}_S)$ satisfies CSP: see formula $(\ast\ast)$ on page 108 of
\cite{Raghunathan}.

Note that taking $r=1$, i.e $S=\{p.\infty\}$ and $\mathcal{O}_S=\Z[\frac{1}{p}]$, makes $BS(1,p^\ell)$ appear
as a finite index subgroup in $\mathbb{G}(\mathcal{O}_S)$, namely
$[\mathbb{G}(\mathcal{O}_S):BS(1,p^\ell)]= 2\ell$.

\begin{proposition}
  Let $\mathbb{H}\subset GL_n$ be a $\Q$-subgroup, and let $\rho:\mathbb{G}\rightarrow\mathbb{H}$ be a
  $\Q$-isomorphism. Then $\rho(BS(1,p^\ell))$ satisfies the CSP.
\end{proposition}

\begin{proof}
  Let $K$ be a finite index subgroup in $\rho(BS(1,p^\ell))$. Since $\mathbb{G}(\mathcal{O}_S)$ satisfies the
  CSP, we find $N>0$ such that $\mathbb{G}(\mathcal{O}_S)(N)\subset \rho^{-1}(K)$. By Lemma 3.1.1.(ii) in
  Chapter I of \cite{MargulisBook}, 
  the subgroup $\rho(\mathbb{G}(\mathcal{O}_S)(N))$ is an $S$-congruence subgroup in
  $\mathbb{H}(\mathcal{O}_S)$, i.e. it contains $\mathbb{H}(\mathcal{O}_S)(M)$ for some $M>1$. Then
  $(\rho(BS(1,p^\ell)))(M)\subset \mathbb{H}(\mathcal{O}_S)(M)\subset
  \rho(\mathbb{G}(\mathcal{O}_S)(N))\subset K$.
\end{proof}

\bibliographystyle{plainnat}
\bibliography{bib.bib}

\end{document}